\documentclass[12 pt]{amsart}
\usepackage{amsmath,amssymb,amsthm,hyperref,adjustbox}
\usepackage{longtable,}
\usepackage{cleveref}
\usepackage{tikz-cd}

\newtheorem{theorem}{Theorem}[section]
\newtheorem{corollary}[theorem]{Corollary}

\newtheorem{lemma}[theorem]{Lemma}

\title[Prime order elements]{Finite simple groups have many classes of prime order elements}
\author{Jessica Anzanello}
\address{Jessica Anzanello\\
Dipartimento di Matematica e Applicazioni\\ University of Milano-Bicocca\\
Via Cozzi 55, 20125 \\Milano, Italy}
\email{j.anzanello@campus.unimib.it}

\author{Pablo Spiga}
\address{Pablo Spiga\\
Dipartimento di Matematica e Applicazioni\\ University of Milano-Bicocca\\
Via Cozzi 55, 20125 \\Milano, Italy}
\email{pablo.spiga@unimib.it}

\begin{document}

\begin{abstract}
Let $T$ be a finite non-abelian simple group. Giudici, Morgan and Praeger have shown that the order of $T$ is bounded above by a function depending on the maximum number of $\mathrm{Aut}(T)$-classes of elements of $T$ of prime-power order. In this note, we strengthen this result by showing, in particular, that “prime-power” can be replaced by “prime”.
\end{abstract}


\maketitle

\section{Introduction}
Let $T$ be a non-abelian simple group and let $x$ be an element of $T$. The $\mathrm{Aut}(T)$-class of $x$ is $\{x^\sigma \mid \sigma \in \mathrm{Aut}(T)\}$. In~\cite{GMP}, Giudici, Morgan and Praeger obtained an upper bound on the order of $T$ in terms of $p$-elements. Indeed, given a prime number $p$, let $m_p(T)$ be the number of $\mathrm{Aut}(T)$-classes of $p$-elements in $T$, and let
\[
m(T) = \max_{p\ \textrm{prime}} m_p(T).
\]
Theorem~1.1 in~\cite{GMP} shows that there exists an increasing function $f:\mathbb{N}\to \mathbb{N}$ such that the order of $T$ is at most $f(m(T))$.

We do not give in this note applications of this result, nor do we discuss its relation to previous remarkable results bounding the order of a finite  group in terms of its conjugacy classes, such as Pyber’s theorem~\cite{Pyber} stating that a group of order $n$ has at least $O(\log n/(\log \log n)^8)$ conjugacy classes. We instead refer the reader to the introductory section of~\cite{GMP}, which provides an excellent account and presents interesting applications of $m(T)$.

In this note, we improve upon the work of Giudici, Morgan and Praeger by considering only elements of prime order. 
 Indeed, given a prime number $p$, let $\mathrm{mpr}_p(T)$ be the number of $\mathrm{Aut}(T)$-classes of elements of order $p$ in $T$, and let
\[
\mathrm{mpr}(T) = \max_{p\ \textrm{prime}} \mathrm{mpr}_p(T).
\]
Clearly, $\mathrm{mpr}_p(T)\le m_p(T)$ and $\mathrm{mpr}(T)\le m(T)$.
\begin{theorem}\label{thrm:1}
There exists an increasing function $f:\mathbb{N}\to \mathbb{N}$ such that, for a finite non-abelian simple group $T$, the order of $T$ is at most $f(\mathrm{mpr}(T))$.
\end{theorem}
As the function $f$ in this result is increasing, we immediately deduce the following.
\begin{corollary}\label{thrm:2}
There exists an increasing function $g:\mathbb{N}\to \mathbb{N}$ such that, for a finite non-abelian simple group $T$, there is a prime $p$  such that $\mathrm{mpr}_p(T)\ge g(|T|)$.
\end{corollary}

Beyond improving upon the work of Giudici, Morgan and Praeger, our results also have an application. Indeed, they constitute one of the key ingredients in establishing a remarkable equivalence between a conjecture of Neumann and Praeger on Kronecker classes in algebraic number fields and a conjecture concerning the clique number of derangement graphs in combinatorics, see~\cite[Section~1.1]{FPS}. We explore this equivalence and apply our results in a forthcoming work. However, in order to obtain this equivalence, we need a strengthening of Theorem~\ref{thrm:1} and Corollary~\ref{thrm:2} for simple groups of Lie type.

Given $x \in T$, we define $\mathrm{mpr}^\ast(x)$ to be the number of generators in $\langle x \rangle$ belonging to distinct $\mathrm{Aut}(T)$-classes. In Lemma~\ref{lemma1}, we give a group-theoretic description of this quantity. In what follows, we let ${\bf o}(x)$ denote the order of $x$. 
Given a prime number $p$, we let
\[
\mathrm{mpr}_p^\ast(T) = \max_{\substack{x \in T \\ {\bf o}(x) = p}} \mathrm{mpr}^\ast(x) \quad \text{and} \quad 
\mathrm{mpr}^\ast(T) = \max_{p\ \textrm{prime}} \mathrm{mpr}_p^\ast(T).
\]
Clearly,
\[
\mathrm{mpr}_p^\ast(T) \le \mathrm{mpr}_p(T) \quad \text{and} \quad \mathrm{mpr}^\ast(T) \le \mathrm{mpr}(T).
\]

When $T$ is the alternating group $\mathrm{Alt}(m)$, we have $\mathrm{mpr}^\ast(T) = 1$, because for every element $x$ of prime order, all non-identity elements of $\langle x \rangle$ are $\mathrm{Aut}(T)$-conjugate to $x$. We show that this behaviour is peculiar to alternating groups.

\begin{theorem}\label{thrm:3}
There exists an increasing function $f:\mathbb{N}\to \mathbb{N}$ such that, for a finite simple group of Lie type $T$, the order of $T$ is at most $f(\mathrm{mpr}^\ast(T))$.
\end{theorem}
As above, since the function $f$ in this result is increasing, we immediately deduce the following.
\begin{corollary}\label{thrm:4}
There exists an increasing function $g:\mathbb{N}\to \mathbb{N}$ such that, for a finite simple group of Lie type $T$, there is an element $x\in T$ having prime order $p$ such that $\mathrm{mpr}^\ast(x)\ge g(|T|)$.
\end{corollary}

Our proof uses similar group-theoretic ideas to those in~\cite{GMP}, but instead of employing classical results on primitive prime divisors, we use Stewart’s work on Lehmer numbers~\cite{Stewart}. Given a positive integer $n$, we denote by $\Phi_n(x)$ the $n^{\mathrm{th}}$ cyclotomic polynomial. %Thus,
%\[
%\Phi_n(x) = \prod_{\substack{j=1 \\ \gcd(j,n)=1}}^n \left(x - \exp\left(\frac{2\pi i j}{n}\right)\right).
%\]
Given $a \in \mathbb{Z}$, we also let $P[\Phi_n(a)]$ denote the largest prime divisor of $\Phi_n(a)$, with the convention that $P[\Phi_n(a)] = 1$ when $\Phi_n(a) = \pm 1$. 
\begin{theorem}[{See~\cite[Theorem~1.1]{Stewart}}]\label{Stewart}
Let $a$ be a prime number. There exists an absolute constant $C$ such that, for $n > C$,
\[
P[\Phi_n(a)] > n \exp\left(\frac{\log n}{104\log\log n}\right).
\]
\end{theorem}The theorem of Stewart is much more general than what we have stated here, but we only need it in the case where $a$ is a prime number, which allows us to take $C$ as an absolute constant in~\cite[Theorem~1.1]{Stewart}.

The structure of the paper is straightforward. We divide the proof of our main results according to whether the non-abelian simple group is alternating, sporadic, classical, or exceptional.  Actually, since there are only finitely many sporadic simple groups, Theorem~\ref{thrm:1} does not require any proof when $T$ is a sporadic simple group.

In the statements of Theorems~\ref{thrm:1} and~\ref{thrm:3}, we insist that the function $f$ is increasing, in order to formulate the results in a form analogous to that of~\cite{GMP}. However, this requirement is not restrictive. Indeed, in Theorem~\ref{thrm:3}, once we have established the existence of a function $f'\colon \mathbb{N} \to \mathbb{R}$ such that 
$|T| \le f'(\mathrm{mpr}^\ast(T))$,
we can define $f\colon \mathbb{N} \to \mathbb{N}$ by 
\[
f(n) = \max\{\lceil f'(x)\rceil \mid x \le n \} + n.
\]
Clearly, $f$ is increasing in $n$, and since $f'(n) \le f(n)$ for all $n$, we also have 
$|T| \le f(\mathrm{mpr}^\ast(T))$.
Hence, the assumption that $f$ is increasing can be made without loss of generality.

\section{Preliminary results}\label{sec:numbertheory}
We start by giving a group-theoretic description of $\mathrm{mpr}^\ast(x)$.
\begin{lemma}\label{lemma1}
Let $G$ be a finite group and let $x\in G$. Then $\mathrm{mpr}^\ast(x)=\varphi({\bf o}(x))/|{\bf N}_{\mathrm{Aut}(G)}(\langle x\rangle):{\bf C}_{\mathrm{Aut}(G)}(x)|$, where $\varphi(n)$ is Euler totient function. 
\end{lemma}
\begin{proof}
We define an equivalence relation on the set of generators of $\langle x\rangle$, where $x_1$ is equivalent to $x_2$ if and only if $x_1^\sigma=x_2$ for some $\sigma\in\mathrm{Aut}(G)$.  Observe that $\langle x\rangle=\langle x_2\rangle=\langle x_1^\sigma\rangle=\langle x_1\rangle^\sigma=\langle x\rangle^\sigma$ and hence $\sigma\in {\bf N}_{\mathrm{Aut}(G)}(\langle x\rangle)$. Therefore, each equivalence class has cardinality $|{\bf N}_{\mathrm{Aut}(G)}(\langle x\rangle):{\bf C}_{\mathrm{Aut}(G)}(x)|$. As $\langle x\rangle$ has $\varphi({\bf o}(x))$ generators, the proof follows.
\end{proof}

\begin{lemma}\label{algebraically closed}
Let $p$ and $s$ be distinct primes, let $d,a$ be natural numbers and  let $x\in\mathrm{PGL}_d(p^a)$ be an element of order $s$ with $\gcd(s,p^a-1)=1$. Then 
$|{\bf N}_{\mathrm{PGL}_d(p^a)}{\langle x\rangle}:{\bf C}_{\mathrm{PGL}_d(p^a)}(x)|\le d$.
\end{lemma}
\begin{proof}
Let $q=p^a$, let $G=\mathrm{PGL}_d(q)$, let $\hat G=\mathrm{GL}_d(q)$ and let $Z$ be the center of $\hat G$. We consider the natural projection $\pi:\hat G\to G$. The preimage of $\langle x\rangle$ has order $s(q-1)$ and is cyclic because $s$ is relatively prime to $q-1$ and because $Z$ is cyclic.  Let $\hat x$ be an element of order $s(q-1)$  projecting to $x$ and let $\hat y=\hat x^{q-1}$. 
 
We first prove that 
\begin{align}\label{eq:29}
|{\bf N}_{\hat G}(\langle \hat y\rangle):{\bf C}_{\hat G}(\hat y)|\le d.
\end{align}
Since $\gcd(s,p)=1$, $\hat y$ is semisimple and hence $\hat y$ is diagonalizable over the algebraic closure $k$ of the field $\mathbb{F}_q$ of cardinality $q$. Let $E\subset k$ be the set of eigenvalues of $\hat y$. As the characteristic polynomial of $\hat y$ has degree $d$, we have $|E|\le d$. We now define a group homomorphism 
$$\varphi:{\bf N}_{\hat G}(\langle \hat y\rangle)\to \mathrm{Sym}(d).$$
If $g\in {\bf N}_{\hat G}(\langle \hat y\rangle)$, then $\hat y^g=\hat y^k$, for some $k\in (\mathbb{Z}/s\mathbb{Z})^{*}$. 
Note that, if $v$ is an eigenvector of $\hat y$ with eigenvalue $\mu$, then $vg$ is an eigenvector with eigenvalue $\mu^{-k}$.
Therefore, $g$ induces a permutation $\varphi(g)$ on $E$ by defining $\mu^{\varphi(g)}=\mu^{-k}$, for all $\mu\in E$. The kernel of of this homomorphism is ${\bf C}_{\hat G}(\hat y)$ and hence we have an embedding of ${\bf N}_{\hat G}(\langle \hat y\rangle)/{\bf C}_{\hat G}(\hat y)$ into $\mathrm{Sym}(d)$. Clearly, this image is cyclic of order at most $d$. This has established~\eqref{eq:29}.

Let $g\in {\bf N}_{G}(\langle x\rangle)$ and let $\hat g\in \hat G$ be an element with $\pi(\hat g)=g$. As $g$ normalizes $\langle x\rangle$, $\hat g$ normalizes $\langle\hat x\rangle$ modulo $Z$ and hence $\hat g$ normalizes $\langle\hat x^{q-1}\rangle=\langle\hat y\rangle$. This has shown that the preimage of ${\bf N}_G(\langle x\rangle)$ via $\pi$ is contained in ${\bf N}_{\hat G}(\langle \hat{y}\rangle)$. Therefore, the preimage of ${\bf N}_G(\langle x\rangle)$ via $\pi$ equals ${\bf N}_{\hat G}(\langle \hat{y}\rangle)$, because the reverse inclusion is clear.
Similarly, if  $g\in {\bf C}_{G}(x)$, then $\hat g$ centralizes  $\hat x$ modulo $Z$ and hence $\hat g$ centralizes $\hat x^{q-1}$ modulo $Z$; however, as ${\bf o}(\hat x^{q-1})$ and $|Z|$ are relatively prime, we deduce that $\hat g$ centralizes $\hat x^{q-1}=\hat y$. Therefore, the preimage of ${\bf C}_G(x)$ via $\pi$ equals ${\bf C}_{\hat G}(\hat{y})$, because the reverse inclusion is clear. Thus
$$|{\bf N}_{G}(\langle x\rangle):{\bf C}_{G}(x)|=|{\bf N}_{\hat G}(\langle \hat y\rangle):{\bf C}_{\hat G}(\hat y)|\le d.\qedhere$$
\end{proof}

\begin{lemma}[{\cite[Equation~(1)]{Siegel}}]\label{siegeltheorem}
Let $f(x)\in\mathbb{Z}[x]$ be a polynomial with at least two distinct roots. Then there exist two positive constants $c_f$ and $c_f'$ depending on $f$ only such that $P[f(q)]\ge c_f\log \log q$ for every $q\ge c_f'$.
\end{lemma}
As far as we are aware, very little is known about estimates for $c_f$ and $c_{f'}$. These seem to constitute a difficult number-theoretic problem. For instance, Keates~\cite[Theorem~I]{keates} shows that when $f$ has degree $2$ or $3$, one may take $c_f = 10^{-7}$ for sufficiently large values of $c_{f'}$ (which are not determined).
	
\section{Proof of Theorem~\ref{thrm:1} for alternating  groups}\label{sec:spoalt}
Let $T$ be the alternating group $\mathrm{Alt}(m)$ of degree $m$, for some $m \ge 5$. The arguments in~\cite[Section~3]{GMP} can be used to obtain a proof of Theorem~\ref{thrm:1} for $\mathrm{Alt}(m)$,  we repeat here some of the ideas.

 When $m \ne 6$, $T$ has $\lfloor m/4 \rfloor$ $\mathrm{Aut}(T)$-classes of involutions, because $\mathrm{Aut}(T)$ is the symmetric group $\mathrm{Sym}(m)$, and two elements in $\mathrm{Sym}(m)$ are conjugate if and only if they have the same cycle structure. When $m = 6$, $T = \mathrm{Alt}(6)$ has one $\mathrm{Aut}(T)$-class of involutions. 

Let $h:\mathbb{N}\to \mathbb{N}$ be defined by $h(m) = \lfloor m/4 \rfloor$. Therefore, 
\[
\mathrm{mpr}(T) \ge \mathrm{mpr}_2(\mathrm{Alt}(m)) \ge h(m),
\]
and hence $
7\,\mathrm{mpr}(T) \ge 7h(m) \ge m$,
where the last inequality follows directly from the definition of $h(m)$. Thus, 
$|T| = m!/2 \le (7\,\mathrm{mpr}(T))!/2$,
and Theorem~\ref{thrm:1} is now proven for alternating groups.

\section{Proof of Theorem~\ref{thrm:3} for classical groups}\label{sec:classical}

Let $a$ and $d$ be integers with $a,d\ge 2$ and let $p$ be a prime divisor of $a^d-1$. Then, $p$ is said to be a primitive prime divisor if $p$ is relatively prime to $a^{k}-1$, for each $k\in \{1,\ldots,d-1\}$. Zsigmondy~\cite{zs} has proved that primitive prime divisors exist, except when $(a,d)=(2,6)$ or when $d=2$ and $a+1$ is a power of $2$.

Let $p$ be a prime number and let $a$ and $d$ be a positive integers with $d\ge 2$. Let $q=p^a$. In this section, we suppose that $T$ is one of the following groups:
\begin{itemize}
\item $T=\mathrm{PSL}_d(q)$, with $(q,d)\notin\{(2,2),(3,2),(4,2),(7,2)\}$,
\item $T=\mathrm{PSp}_d(q)$, with $d$ even, $d\ge 4$ and $(q,d) \ne (2,4)$,
\item $T=\mathrm{PSU}_d(q)$, with $d\ge 3$ and $(q,d)\ne (2,3)$,
\item $T=\Omega_d(q)$, with $d\ge 7$ and $dp$ odd,
\item $T=\mathrm{P}\Omega_d^\pm(q)$, with $d\ge 8$ and $d$ even.
\end {itemize}
We also let $\delta=2$ when $T$ is unitary, and $\delta=1$ otherwise.

When $T=\mathrm{PSL}_d(q)$, $T=\mathrm{PSU}_d(q)$ with $d$ odd,  $T=\mathrm{PSp}_d(q)$ or $T=\mathrm{P}\Omega_d^-(q)$, we let $s$ be the largest primitive prime divisor of $p^{ad\delta}-1$. The existence of $s$ is guaranteed by Zsigmondy's theorem~\cite{zs}, except when $T=\mathrm{PSL}_2(p)$ and $p$ is a Mersenne prime, and except when $T\in\{\mathrm{PSL}_2(8),\mathrm{PSL}_3(4),\mathrm{PSU}_3(4),\mathrm{PSL}_6(2),\mathrm{PSp}_6(2)\}$. In the proof of Theorem~\ref{thrm:3}, we may exclude any finite number of cases and hence we may exclude the groups where $(p,ad)=(2,6)$ from further considerations. Whereas, we will consider the case $T=\mathrm{PSL}_2(p)$ and $p$ is a Mersenne prime at the end of the proof.

When $T=\mathrm{PSU}_d(q)$ with $d$ even or when $T=\Omega_d(q)$ with $dq$ odd, we let $s$ be the largest primitive prime divisor of $p^{a(d-1)\delta}-1$. The existence of $s$ is guaranteed by Zsigmondy's theorem~\cite{zs}, with the exception of $\mathrm{PSU}_4(2)$.

When $T=\mathrm{P}\Omega_d^+(q)$, we let $s$ be the largest primitive prime divisor of $p^{a(d-2)}-1$. Again, the existence of $s$ is guaranteed by Zsigmondy's theorem~\cite{zs}, except when $T=\mathrm{P}\Omega_8^+(2)$. Clearly, as above, for the proof of Theorem~\ref{thrm:3}, we may exclude the group $\mathrm{PSU}_4(2)$ and $\mathrm{P}\Omega_8^+(2)$ from further considerations.

Let $B=\mathrm{Aut}(T)\cap \mathrm{P}\Gamma\mathrm{L}_d(q^\delta)$. From~\cite[page~xvi]{atlas} and~\cite[Chapter~2]{kl}, we have \begin{equation}\label{eq:28}|\mathrm{Aut}(T):B|\le 6.\end{equation} 

Now, let $x\in T$ with ${\bf o}(x)=s$ and observe that $\gcd(s,q^\delta-1)=1$. As $B\le \mathrm{P}\Gamma\mathrm{L}_d(q^\delta)$, we have
$$|{\bf N}_{B}(\langle x\rangle):{\bf C}_{B}(x)|\le |{\bf N}_{\mathrm{P}\Gamma\mathrm{L}_d(q^\delta)}(\langle x\rangle):{\bf C}_{\mathrm{P}\Gamma\mathrm{L}_d(q^\delta)}(x)|.$$
Thus, from~\eqref{eq:28} and from Lemma~\ref{algebraically closed}, we deduce
$$|{\bf N}_{\mathrm{Aut}(T)}(\langle x\rangle):{\bf C}_{\mathrm{Aut}(T)}(x)|\le 6d\delta a.$$
Now, Lemma~\ref{lemma1} implies
\begin{equation}\label{ppd}
\mathrm{mpr}^\ast(x)\ge \frac{s-1}{6a\delta d}.
\end{equation} 

As $s$ is a primitive prime divisor, we deduce that $s$ is a prime divisor of either $\Phi_{ad\delta}(p)$, $\Phi_{a(d-1)\delta}(p)$, or $\Phi_{a(d-2)}(p)$, depending on the type of $T$. Let $b$ be either $ad\delta$, $a(d-1)\delta$, or $a(d-2)$ such that $s$ is a divisor of $\Phi_b(p)$.

Let $C$ be the constant appearing in Theorem~\ref{Stewart} and let $C'=\max(C,e^e)$. If $b>C'$, then 
\begin{align*}
\mathrm{mpr}^\ast(T)&\ge \frac{b\exp\left(\frac{\log(b)}{104\log \log (b)}\right)-1}{6ad\delta}\ge \frac{b}{6ad\delta}\exp\left(\frac{\log(b)}{104\log\log (b)}\right)-1\\
&\ge\frac{1}{12}\exp\left(\frac{\log(b)}{104\log\log (b)}\right)-1,
\end{align*}
where the last inequality follows from the fact that $b/6ad\delta\ge 1/12$.
This yields $\exp(\log(b)/104\log\log(b))\le 12(\mathrm{mpr}^\ast(T)+1)$ and hence $$\frac{\log (b)}{\log\log(b)}\le 104\log(12(\mathrm{mpr}^\ast(T)+1)).$$
As $x\mapsto \log x/\log\log x$ is an increasing function for $x>e^e$, we deduce 
$b\le f_1(\mathrm{mpr}^\ast(T))$, for some function $f_1:\mathbb{N}\to\mathbb{R}$. In particular, as $ad\delta\le 2b$, we get $ad\delta\le 2f_1(\mathrm{mpr}^\ast(T))$. If we let $f_2(x)=2f_1(x)+C'$, then $ad\delta\le f_2(x)$, regardless of whether $b>C'$ or $b\le C'$. 

Assume $b\ge 3$. For each $n\in\{3,\ldots,b\}$,  let $c_{\Phi_n}$ and $c_{\Phi_n}'$ be the constants appearing in Lemma~\ref{siegeltheorem} applied with $f=\Phi_n$ and let 
\begin{align*}
c&=\min\{c_{\Phi_n}\mid 3\le n\le f_1(\mathrm{mpr}^\ast(T))\},\\
c'&=\max\{c_{\Phi_n}'\mid 3\le n\le f_1(\mathrm{mpr}^\ast(T))\}.
\end{align*}In particular, $c$ and $c'$ are two constants depending only on $\mathrm{mpr}^\ast(T)$. If $p\ge c'$, then  from Lemma~\ref{siegeltheorem}, we have $s\ge c\log\log p$. Thus from~\eqref{ppd} we deduce
\begin{align*}
c\log\log p-1&\le s-1\le 6\mathrm{mpr}^\ast(T)ad\delta\le 6\mathrm{mpr}^\ast(T)f_2(\mathrm{mpr}^\ast(T)).
\end{align*}
Thus $p\le \tilde f_3(\mathrm{mpr}^\ast(T))$, where $$\tilde f_3(x)=\frac{\exp(\exp( 6xf_2(x)+1))}{c}.$$
If we let $f_3(x)=\max\{c',\tilde f_3(x)\}$, then $p\le f_3(\mathrm{mpr}^\ast(T))$, regardless of whether $p\ge c'$ or $p<c'$. As $|T|\le q^{\delta d^2}\le p^{2ad^2}$, we deduce that $|T|\le f_4(\mathrm{mpr}^\ast(T))$, where $f_4(x)=f_3(x)^{2(f_2(x))^2}$.

Assume now $b=2$, that is, $b=ad$, $a=1$, $d=2$ and $T=\mathrm{PSL}_2(p)$. We are including in this case also the possibility that $p^2-1$ admits no primitive prime divisor, that is, $p$ is a Mersenne prime.
 From Lemma~\ref{siegeltheorem}, there exist two absolute constants $v$ and $v'$ such that $P[p^2-1]\ge v\log \log p$ when $p\ge v'$. Let $v''=\max\{3,v',\exp(\exp(2/v))\}$. If $p\le v''$, then $|T|=p(p^2-1)/\gcd(2,p-1)$ is bounded above by the absolute constant $v''^3$. Assume then $p> v''$ and let $s=P[p^2-1]$. Thus $p>v'>2$ and $s>2$. Now, if $s$ divides $p+1$, Lemmas~\ref{lemma1} and~\ref{algebraically closed} imply $\mathrm{mpr}^\ast(T)\ge (s-1)/2\ge (v\log\log p-1)/2$. Thus $p$ is bounded above by a function of $\mathrm{mpr}^\ast(T)$ and hence so is $|T|.$ If $s$ divides $p-1$, let $x\in T=\mathrm{PSL}_2(p)$ be the projective image of the matrix
\[
\begin{pmatrix}
\zeta&0\\
0&\zeta^{-1}
\end{pmatrix}
,\]
where $\zeta\in\mathbb{Z}/p\mathbb{Z}$ has order $s$ in the multiplicative group of $\mathbb{Z}/p\mathbb{Z}$.
A direct computation in $\mathrm{PSL}_2(p)$ (using ${\bf o}(\zeta)>2$) gives ${\bf N}_{\mathrm{Aut}(T)}(\langle x\rangle)$ is a dihedral group of order $2(p-1)$ and ${\bf C}_{\mathrm{Aut}(T)}(x)$ is a cyclic group of order $p-1$. Therefore, from Lemma~\ref{lemma1}, we have $\mathrm{mpr}^\ast(x)\ge (s-1)/2\ge (c\log\log p-1)/2$. As above, we deduce that $p$ and  hence $|T|$ are bounded above by a function of $\mathrm{mpr}^\ast(T)$.

\section{Proof of Theorem~\ref{thrm:3} for exceptional groups of Lie type}\label{sec:exceptional}

Let $T$ be an exceptional group of Lie type and let $s=P[\phi_{ab}(p)]$, where $b\ge 3$ and $s$ divides $|T|$. The existence of $b$ and hence of $s$ can be deduced from the formulae of the order of the exceptional groups of Lie type, see for instance~\cite[page~xvi]{atlas}. For example, when $T=E_8(p^a)$, we may take $b=30$, because $p^{30a}-1$ is a divisor of the order of $T$.

Let $x\in T$ with ${\bf o}(x)=s$ and let $b$ be the minimum dimension of a faithful linear representation of $T$ over a field of cardinality $p^a$. Observe that $b$ is bounded above by the Lie rank of $T$ and hence it is bounded above by an absolute constant. For instance, when $T=E_8(p^a)$, we may take $b=248$, see for istance~\cite[page~200]{kl}.

In particular, $T$ admits an embedding in $\mathrm{PGL}_b(p^a)$. Clearly, $|{\bf N}_T(\langle x\rangle):{\bf C}_T(x)|\le |{\bf N}_{\mathrm{PGL}_{b}(p^a)}(\langle x\rangle):{\bf C}_{\mathrm{PGL}_{b}(p^a)}(x)|\le b$, where the last inequality follows from Lemma~\ref{algebraically closed}. Thus, from Lemma~\ref{lemma1}, we have $$\mathrm{mpr}^\ast(T)\ge \frac{s-1}{|\mathrm{Out}(T)|b}\ge \frac{s-1}{2ab},$$ where the last inequality follows from~\cite[Table~5, page~xvi]{atlas}. 
We may now argue as for the classical groups. Indeed, using Theorem~\ref{Stewart}, we may bound $a$ above as a function of $\mathrm{mpr}^\ast(T)$ and, once $a$ is bounded, we may use Lemma~\ref{siegeltheorem} to bound $p$ from above as a function of $\mathrm{mpr}^\ast(T)$.

\section*{Acknowledgements}
 The second author is funded by the European Union via the Next
Generation EU (Mission 4 Component 1 CUP B53D23009410006, PRIN 2022, 2022PSTWLB, Group
Theory and Applications).

\end{document}